\documentclass[reqno]{amsart}

\usepackage[utf8]{inputenc}
\usepackage{amsmath,amsthm,amssymb,amsfonts,enumitem}
\usepackage[all]{xy}
\usepackage{mathtools}
\usepackage{fancyhdr}
\usepackage{hyperref}
\usepackage{dirtytalk}
\usepackage{faktor}

\newtheoremstyle{stile}
	{3pt}
	{3pt}
	{\itshape}
	{.5cm}
	{\scshape}
	{:}
	{1mm}
	{}
\theoremstyle{stile}
\newtheorem{defin}{Definition}[section]

\newtheorem{thm}[defin]{Theorem}
\newtheorem{lem}[defin]{Lemma}

\newtheorem{cor}[defin]{Corollary}

\theoremstyle{remark}
\newtheorem{oss}[defin]{Remark}

\newtheorem{osss}[defin]{Remarks}

\newcommand{\Q}{\mathbb{Q}}

\newcommand{\A}{\mathbb{A}}
\renewcommand{\O}[1]{\mathcal{O}_{#1}}
\newcommand{\id}[1]{\mathfrak{#1}}
\renewcommand{\phi}{\varphi}
\newcommand{\un}[1]{\underline{#1}}

\DeclareMathOperator{\Spec}{Spec}
\DeclareMathOperator{\Frac}{Frac}
\DeclareMathOperator{\height}{ht}
\DeclareMathOperator{\spec}{sp}
\DeclareMathOperator{\trdeg}{trdeg}

\makeatletter
\@namedef{subjclassname@2020}{%
  \textup{2020} Mathematics Subject Classification}
\makeatother
\makeatletter
\newcommand{\leqnomode}{\tagsleft@true\let\veqno\@@leqno}
\newcommand{\reqnomode}{\tagsleft@false\let\veqno\@@eqno}
\makeatother

\begin{document}

\title{Hilbert specialization of parametrized varieties}
\author{Angelo Iadarola}
\date{}
\email{angelo.iadarola@univ-lille.fr}
\address{Univ. Lille, CNRS, UMR 8524 - Laboratoire Paul Painlevé, F-59000 Lille, France}
\begin{abstract}
Hilbert specialization is an important tool in Field Arithmetic and Arithmetic Geometry, which has usually been intended for polynomials, hence hypersurfaces, and at
scalar values. In this article, first, we extend this tool to prime ideals, hence affine varieties, and
offer an application to the study of the irreducibility of the intersection of varieties. Then, encouraged by recent results,
we consider the more general situation in which the specialization is done at polynomial values, instead of scalar values.
\end{abstract}

\pagestyle{fancy}
\fancyhf{}
\chead{\footnotesize{Hilbert specialization of parametrized varieties}}
\rhead{\footnotesize{\thepage}}

\maketitle
\thispagestyle{empty}

\section{Introduction}
Hilbert Irreducibility Theorem has been a core result in Field Arithmetic for many decades. A simple form, see for example \cite[Page 218]{fj}, says that, given an irreducible polynomial
$P(T,Y)$ in $\Q(T)[Y]$, one can find infinitely many $t\in\Q$ such that the so-called {\itshape
specialized} polynomial $P(t,Y)$ is irreducible in $\Q[Y]$.\par
This result has been generalized under many aspects. First of all, the notion of {\itshape Hilbertian 
field} has been introduced to identify all those fields $K$ for which the previous statement is verified
in the case of polynomials, which are separable in $Y$, if we replace $\Q$ with $K$, see for example \cite[Page 218]{fj}.
Moreover, if $K$ is of characteristic $0$ or imperfect, the same result holds if we replace a separable irreducible polynomial in two variables $P(T,Y)$ in $K(T)[Y]$ with several irreducible
polynomials $P_1(\un{T},\un{Y}),\ldots,P_n(\un{T},\un{Y})$ in $K(\un{T})[\un{Y}]$ in two arrays of variables, $\un{T}=(T_1,\ldots,T_r)$ and $\un{Y}=(Y_1,\ldots,Y_s)$ for
$r,s$ positive integers: we can find a Zariski-dense subset $H\subset\A_K^r$ such that the polynomial $P_i(\un{t},\un{Y})$
is irreducible in $K[\un{Y}]$ for every $\un{t}\in H$ and $i=1,\ldots,n$, see for example \cite[Section 12.1]{fj}.\par
If we look at this statement from a geometric point of view, another potential generalization arises
naturally. Giving an irreducible polynomial $P(\un{T},\un{Y})$ in $K(\un{T})[\un{Y}]$ is equivalent
to giving an irreducible $K(\un{T})$-hypersurface\footnote{The precise definition will be given in the following
section.} $V_{K(\un{T})}(P(\un{T},\un{Y}))$ in the $s$-dimensional affine space $\A_{K(\un{T})}^s$ over
$K(\un{T})$. In these terms, for $K$ Hilbertian and
of characteristic $0$ or imperfect, Hilbert Irreducibility says that for
a Zariski-dense set of choices $\un{t}$ of the variables $\un{T}$ in $K^r$, the {\itshape specialized} algebraic set
$V_K(P(\un{t},\un{Y}))\subset\A_K^s$ is an irreducible
$K$-hypersurface. It is then natural to ask if an analogous result holds in the case of
a $K(\un{T})$-variety of codimension bigger than $1$. In algebraic terms this is equivalent to
finding values $\un{t}$ in $K^r$ of $\un{T}$ such that a nonzero prime ideal
$\id{p}_T$ in $K(\un{T})[\un{Y}]$ remains a nonzero prime ideal of
$K[\un{Y}]$ when {\itshape specializing} $\un{T}$ at $\un{t}$.
This is indeed one of the problems we are addressing in this article.\par
We want to make another step further. Until now the Hilbert {\itshape specialization} has typically
been intended for scalar values in $K^r$. However, there is a recent result, see
\cite{bdn}, in which a polynomial version of the Schinzel hypothesis is proved by specializing at polynomial values instead of scalar values. Given irreducible polynomials $P_i(\un{T},\un{Y})$, for
$i=1,\ldots,n$, in a polynomial ring
$R[\un{T},\un{Y}]$ for $R$ an integral domain,
the variables $\un{T}$ are replaced by polynomials $\un{Q}(\un{Y})=(Q_1(\un{Y}),\ldots,Q_r(\un{Y}))$,
in the other variables and, under appropriate assumptions, all the {\itshape specialized} polynomials $P_i\left(\un{Q}(\un{Y}),\un{Y}\right)$
are shown to remain irreducible in $R[\un{Y}]$. This recent development encourages to pursue the study of the specialization at polynomials. Indeed, another goal
of this article will be to apply the results of the first part to obtain a more general version of them
where the {\itshape specialization} of the variables $\un{T}$ occurs at polynomials in
$(K[\un{Y}])^r$ instead of scalars in $K^r$.
\subsection{Notation and main results}
In this paper, given a field $F$, a set of variables $\un{X}=(X_1,\ldots,X_m)$ and
polynomials $f_1,\ldots,f_n$ in $F[\un{X}]$, we denote by $V_F(f_1,\ldots,f_n)$ the {\itshape affine subvariety} of $\A^m_{\Bar{F}}$ of
equations $\{f_i(x_1,\ldots,x_m)=0,\,i=1,\ldots,n\}$. More formally, following \cite[Definition 3.4.7]{liu},
it is the  affine
scheme associated to the finitely generated $F$-algebra $\faktor{F[\un{X}]}{\langle f_1,\ldots,f_n\rangle}$. As,
by extension of scalars, the same set of elements gives rise to different varieties over different fields, to avoid any ambiguity, we use the word {\itshape $F$-variety} to specify the base field.
If the ideal $\langle f_1,\ldots,f_n\rangle$ is prime in $F[\un{X}]$, we say that the $F$-variety is {\itshape irreducible}.
Moreover, if the $F$-variety has codimension\footnote{We define the dimension of a variety as
the Krull dimension of the ring $\faktor{F[\un{X}]}{\langle f_1,\ldots,f_n\rangle}$.} $1$, we call it an $F$-{\itshape hypersurface}.
Finally we say that an irreducible $F$-variety is {\itshape separable} if $\Frac\left(\faktor{F[\un{X}]}{\langle f_1,\ldots,f_n\rangle}\right)$ is a separable extension of $F(\un{X})$, in the general sense as in \cite[Lemma 2.6.1]{fj}.\par
We state the first result of the article.
\begin{thm}\label{main1}
Let $K$ be a Hilbertian field, $\un{P}(\un{T},\un{Y})=\{P_1(\un{T},\un{Y}),\ldots,P_l(\un{T},\un{Y})\}$ a set of polynomials in $K[\un{T},\un{Y}]$ such that $V_T=V_{K(\un{T})}(\un{P}(\un{T},\un{Y}))$ is
a separable irreducible $K(\un{T})$-variety. Then for every $\un{t}=(t_1,\ldots,t_r)\in K^r$ in some Zariski-dense subset of $\A_K^r$,
the $K$-variety $V_t=V_K(\un{P}(\un{t},\un{Y}))\subset\A_{K}^s$, where $\un{P}(\un{t},\un{Y})$ is the set made of the specialized polynomials at $\un{t}$, is irreducible and its dimension $\dim_K V_t$ is equal to $\dim_{K(\un{T})} V_T$, the dimension of $V_T$ as a $K(\un{T})$-variety.
\end{thm}
\begin{osss}
   (a) Even if the polynomials in $\un{P}$ are irreducible, it is not enough that the corresponding
   specialized polynomials $P_i(\un{t},\un{Y})$ are
irreducible to conclude that the variety $V_K(\un{P}(\un{t},\un{Y}))$ is irreducible. Generally speaking,
it may be that an ideal is generated by irreducible polynomials in $K[\un{Y}]$ but is not a prime
ideal. Take for example the ideal $\langle Y-X,Y-X^2\rangle$, which contains the product
$X(X-1)$.\par\smallskip
    (b) Hilbert specialization in the case of polynomials can be also performed over rings instead of fields, as in \cite[Theorem 1.6 and Remark 4.4]{bdkn}. It is
    then natural to ask if we can extend Theorem \ref{main1} to a ring $R$ as well, at least when $R$ is a Unique Factorization Domain.
    As our proof and \cite{bdkn} suggest, it may be possible that such a version holds if a nonzero
    element $\phi\in R$ is inverted, i.e. if $R$ is replaced by $R[\phi^{-1}]$ (a restriction that cannot
    be avoided in general). This, however, remains unclear at the moment.\par\smallskip
(c) Theorem \ref{main1} can be seen as a Hilbertian version of Bertini's Theorem. If we look at
the statement of Bertini's Theorem given in \cite[Corollary 10.4.3]{fj}, we can see the similarity
between the two statements. However, the two statements go on parallel routes. Bertini's
Theorem demands an algebraically closed base field $K$. In particular the case $s=1$ is excluded in the Bertini context while it is a significant situation in the Hilbert context of Theorem \ref{main1}.
\end{osss}
A recurrent tool in the article will be the {\itshape generic} polynomial. Given an integer $D\geq 0$ we define the {\itshape generic}
polynomial of degree $D$:
\[\mathcal{Q}_D(\un{\Lambda},\un{Y})=\sum_{i=1}^{N_D} \Lambda_i\hspace{1pt} Q_i(\un{Y})\]
where $Q_i(\un{Y})$ varies over all the power products $Y_1^{\beta_1}\cdots Y_s^{\beta_s}$, $\beta_i\geq0$, in the variables $\un{Y}$ of degree smaller or equal than $D$ and $N_D$ is the number
of such power products and $\un{\Lambda}=(\Lambda_1,\ldots,\Lambda_{N_D})$ is a set of auxiliary variables, which correspond to the \say{generic} coefficients.\par
The application of Theorem \ref{main1}, with these variables $\un{\Lambda}$ playing the role of the variables $\un{T}$ in the statement of Theorem \ref{main1}, is the main tool of the proof of the next two statements.\par
The first one is about the intersection between an irreducible $K$-variety and a \say{generic} $K(\un{T})$-hypersurface.
\begin{thm}\label{main2}
Let $K$ be a Hilbertian field of characteristic $0$. For $l\geq1$, let $P_1(\un{Y}),\ldots,P_l(\un{Y})$ be polynomials
in $K[\un{Y}]$ such that $V=V_K(P_1,\ldots,P_l)$ is an irreducible $K$-variety of positive dimension $d$.
Then, for every $\un{\lambda}\in K^{N_D}$ in a Zariski-dense subset of $\A_K^{N_D}$, the $K$-variety $V\cap V_K(\mathcal{Q}_D(\un{\lambda},\un{Y}))$ is irreducible and $\dim_K V\cap V_K(\mathcal{Q}_D(\un{\lambda},\un{Y})))=d-1$.
\end{thm}
We will see that a more general result, in fact, holds. If we replace the \say{generic} hypersurface by an intersection of
$\rho$ \say{generic} hypersurfaces, with $\rho\leq d$, the intersection with the original
variety $V$ will \say{often} be an irreducible $K$-variety of dimension $d-\rho$, see Corollary \ref{main21}.\par
The generic polynomial will be central also in the proof of the last main result of this paper, where
the goal is to generalize Theorem \ref{main1} to the
situation in which the variables are specialized at polynomials.\par
We note that the set $K[\un{Y}]_D$ of all the polynomials of degree smaller or equal than $D$ can
be endowed with a Zariski topology through the natural isomorphism with $\A^{N_D}_K$ which associates to
a polynomial $P(\un{Y})$ the point in $\A^{N_D}_K$ having the coefficients of $P$ as coordinates.
\begin{thm}\label{main3}
Let $K$ be a Hilbertian field of characteristic $0$. For $l\geq1$, let
$P_1(\un{T},\un{Y}),\ldots,P_l(\un{T},\un{Y})$ be polynomials
in $K[\un{T},\un{Y}]$ such that $V_T=V_{K(\un{T})}(P_1,\ldots,P_l)$ is an irreducible $K(\un{T})$-variety of dimension $d$.
Fix non-negative integers $D_1,\ldots,D_r$.
Then for every $\un{U}=\left(U_1(\un{Y}),\ldots,U_r(\un{Y})\right)$ in a Zariski-dense subset of $\prod_{i=1}^r K[\un{Y}]_{D_i}$,
the $K$-variety $V_U=V_K(P_1(\un{U},\un{Y}),\ldots,P_l(\un{U},\un{Y}))$ is an irreducible $K$-variety of
dimension $\dim_K V_U=d$.
\end{thm}
\begin{osss}\label{rem3}
(a) The case $l=r=1$ (i.e. one polynomial and one variable $T$) yields the Schinzel
hypothesis for the polynomial ring $K[\un{Y}]$, as stated in \cite[Section 1.1]{bdn}: given an irreducible polynomial $P(T,\un{Y})$
in $K[T,\un{Y}]$, for every $U(\un{Y})$ in some Zariski-dense subset of $K[\un{Y}]_D$, the polynomial
$P(U(\un{Y}),\un{Y})$ is irreducible in $K[\un{Y}]$.\par
(b) By taking $D_i=0$, for every $i$, Theorem \ref{main3} implies Theorem \ref{main1} in characteristic $0$
(see Remark \ref{3imp1}).
\end{osss}
\subsection{Hilbert sets}
In this introduction, we have restricted the choice of the base field to a Hilbertian field.
However, these statements can be generalized to every field thanks to the notion of {\itshape Hilbert
sets}. We are giving only a quick review on this topic; refer to \cite[Sections 12,13]{fj} for
more details.\par
Given a set of irreducible polynomials $P_1(\un{T},\un{Y}),\ldots,P_l(\un{T},\un{Y})$ in $K(\un{T})[\un{Y}]$, we define the following set:
\[H_K(P_1,\ldots,P_l)=\{\un{t}\in K^r|f_i(\un{t},\un{Y}) \text{ is irreducible in }K[\un{Y}]\text{ for
each }i=1,\ldots,l\}\text{.}\]\par
Such sets and their intersections with non-empty open Zariski-subsets are called {\itshape Hilbert subsets} of $K^r$. Alternatively, if we do not specify the dimension $r$,
we say that $H_K(P_1,\ldots,P_l)$ is a {\itshape Hilbert set} of $K$ to say that it is a Hilbert
subset of $K^r$ for some $r$. This definition does not require any hypothesis on the field.\par
These sets can be empty: take for example $K=\mathbb{C}$ and $P(T,Y)\in\mathbb{C}(T)[Y]$, an irreducible complex monic polynomial, such that $\deg_Y P>1$. \par
Therefore
we define the {\itshape Hilbertian fields}, which we have already mentioned, as the fields for which all
these sets are Zariski-dense in $\A_K^r$. Of course Hilbertian fields are the most convenient setting
because the sets of elements for which our results hold are generally as big as possible.\par
By definition, Hilbert sets are stable under finite intersection. 
We will show, in the next sections, that the Zariski-dense subsets involved in Theorems \ref{main1}, \ref{main2} and
\ref{main3} are, in fact, Hilbert sets, with no assumption on the field $K$. Consequently, our main
results hold, in fact, for a finite number of varieties/ideals.\par
In particular, the original Hilbert irreducibility property in its full form, i.e. for several polynomials
$P_1,\ldots,P_l$, follows from Theorem \ref{main1}: just apply it to each of the prime ideals $\langle P_i\rangle$ and then take the intersection of the Hilbert sets.\par
The paper is organized as follows. In Section 2 we will focus on Theorem \ref{main1}: we will see the two steps of its proof and some further remarks. Some more preliminary tools will be added
when convenient. In Section 3 we will first give some results about generic polynomials and
then, finally, we will use them to obtain the proofs of Theorems \ref{main2} and \ref{main3}.
\subsection*{Acknowledgements}
I would like to thank the Laboratoire Paul Painlevé for giving me the possibility to pursue a PhD.
Moreover I would like to thank Pierre Dèbes for his precious guidance and infinite patience. Finally, I
would like to thank Lorenzo Ramero for some useful suggestions concerning Lemma \ref{ramero}.

\section{Proof of Theorem \ref{main1}}
The cornerstone of the article is the proof of Theorem \ref{main1}, which will be
essential to prove Theorems \ref{main2} and \ref{main3}. The statement we are actually going to prove
hereinafter is a more general version of Theorem \ref{main1}.
\begin{thm}\label{main11}
Let $K$ be a field. Assume that
$\id{p}_T=\langle P_1(\un{T},\un{Y}),\ldots,P_l(\un{T},\un{Y})\rangle$ is a prime ideal in $K[\un{T},\un{Y}]$ such that $\id{p}_T\cap K[\un{T}]=\{0\}$ and $\Frac \left(\faktor{K[\un{T},\un{Y}]}{\id{p}_T}\right)$ is separable
over $K(\un{T})$. Denote by $d$ the dimension of $\faktor{K[\un{T},\un{Y}]}{\id{p}_T}$ as a $K[\un{T}]$-algebra.
Then for every $\un{t}$ in a Hilbert subset of $K^r$, the following equivalent statements hold:
\begin{enumerate}[label={\normalfont (\roman*)}]
    \item The quotient $\faktor{K[\un{Y}]}{\id{p}_t}$
is an integral algebra of dimension $d$ over $K$, where $\id{p}_t=\langle P_1(\un{t},\un{Y}),\ldots,P_l(\un{t},\un{Y})\rangle$ is the specialized ideal.
    \item The ideal $\id{p}_t$ is prime and its height $\height \id{p}_t$ is equal to the height $\height \id{p}_T$ of $\id{p}_T$.
    \item The $K$-variety $V_t=V(\un{P}(\un{t},\un{Y}))$, where $\un{P}(\un{t},\un{Y})$ is the set made of the specialized polynomials at $\un{t}$, is irreducible and its dimension $\dim_K V_t$ is equal to $\dim_{K(\un{T})} V_T$, the dimension of $V_T$ as a $K(\un{T})$-variety.
\end{enumerate}
\end{thm}
\begin{oss}
Consider the extreme case for which the ideal $\id{p}_T$ in Theorem \ref{main11} is
maximal as an ideal in $K(\un{T})[\un{Y}]$. Then the quotient $\faktor{K(\un{T})[\un{Y}]}{\id{p}_T}$ is
an algebraic separable field extension of $K(\un{T})$ of finite degree. In this case, Theorem \ref{main11} implies
the well-known fact that the degree of the extension is preserved under specialization at every $\un{t}$
in a Hilbert set of $K$. Moreover, if the extension is Galois, then also the Galois group of the
extension is preserved. The study of the specialization of Galois extensions is central in Inverse
Galois Theory, see for example \cite{volklein,fj}.
\end{oss}
The equivalence between the three statements is easy. The proof is given right after the following lemma, which we will frequently use in the paper.
\begin{lem}\label{fractions}
Let $K$ be a field. Then, given a prime ideal \[\id{p}=\langle P_1(\un{T},\un{Y}),\ldots,P_l(\un{T},\un{Y})\rangle\subset K[\un{T},\un{Y}]\]
such that $\id{p}\cap K[\un{T}]=\{0\}$, the ideal
\[\tilde{\id{p}}=\langle P_1(\un{T},\un{Y}),\ldots,P_l(\un{T},\un{Y})\rangle\subset K(\un{T})[\un{Y}]\]
is prime and its height $\height\tilde{\id{p}}$ is equal to the height $\height\id{p}$ of $\id{p}$.
\end{lem}
\begin{proof}
Denoting
$S=K[\un{T}]\setminus \{0\}$, we remark that $S$ is a multiplicative subset and $S^{-1}K[\un{T},\un{Y}]=K(\un{T})[\un{Y}]$.
The natural morphism sending an element $a$ in $K[\un{T},\un{Y}]$ to $\frac{a}{1}$ in $S^{-1}K[\un{T},\un{Y}]$ induces a bijective correspondence between prime ideals in $K[\un{T},\un{Y}]$ having empty intersection with $S$ and prime ideals in $K(\un{T})[\un{Y}]$ \cite[Proposition 3.11(iv)]{am}.
So we can consider the prime ideal $\tilde{\id{p}}$ associated to $\id{p}$ by this correspondence, which
is the ideal generated by the image of $\id{p}$ under the aforementioned morphism: the ideal
\[\tilde{\id{p}}=\langle P_1(\un{T},\un{Y}),\ldots,P_l(\un{T},\un{Y})\rangle\subset K(\un{T})[\un{Y}]\] is prime.\par
Now we want to check that the height is preserved by this extension. Consider a maximal chain of primes in $\id{p}$ in $K[\un{T},\un{Y}]$,
\[\id{p}_1\subset\id{p}_2\subset\ldots\subset\id{p}\text{.}\]\par
As $\id{p}_i\subset\id{p}$ for each $i$, we have $\id{p}_i\cap S=\emptyset$, so $\tilde{p}_i$ is prime in
$K(\un{T})[\un{Y}]$. Since the inclusions are
conserved for $\tilde{\id{p}}_i$, we have $\height\id{p}\leq \height\tilde{\id{p}}$. Vice versa,
assume that a chain of primes $\tilde{\id{p}}_i$ inside $\tilde{\id{p}}$ is longer than $\height\id{p}$: by
the previous correspondence we can build a chain of ideals inside $\id{p}$ longer than
$\height\id{p}$, which is a contradiction. So $\height\id{p}=\height\tilde{\id{p}}$.
\end{proof}

We can now give the proof of the equivalence between the three statements of Theorem \ref{main11}.
\begin{proof}
(i)$\Leftrightarrow $(ii) It easily follows from the fact that the height of the ideal is equal to the codimension of the quotient algebra by the ideal.
\par
(ii)$\Leftrightarrow$(iii) By Lemma \ref{fractions}, $\height\id{p}_T=\height\tilde{\id{p}}_T$. By statement (ii), $\height\id{p}_t=\height\id{p}_T$. So $\height\id{p}_t=\height\tilde{\id{p}}_T$.
As the height of a prime ideal is the codimension of the associated variety and the rings
$K(\un{T})[\un{Y}]$ and $K[\un{Y}]$ have the same Krull dimension, statement (iii) follows. The converse
is easily shown in the same manner.
\end{proof}
There are two requirements for statement (i) of Theorem \ref{main11}: we
want $\faktor{K[\un{T},\un{Y}]}{\id{p}_T}$ to be integral and of the correct dimension. We
are going to prove the two parts separately: first we show that each integral component of $\faktor{K[\un{Y}]}{\id{p}_t}$ is of dimension $d$, then that there is only one such component.
\subsection{First part}\label{first}
This part is mostly geometric. We are quickly recalling some tools that we are going to use. The
statements are taken from \cite{groth}.
\begin{lem}[Local freeness, Lemma 5.11]\label{freeness}
Let $A$ be a Noetherian domain and $B$ a finite-type $A$-algebra. Let $M$ be a finite $B$-module.
Then there exists $c\in A$, $c\neq 0$ such that the localisation $M\left[c^{-1}\right]$ is a free
module over $A\left[c^{-1}\right]$.
\end{lem}
This result, due to Grothendieck, has a deep consequence, the so-called {\itshape Generic flatness}.
\begin{thm}[Generic flatness, Theorem 5.12]\label{flatness}
Let $S$ be a Noetherian and integral scheme. Let $p:X\to S$ be a finite type morphism and let $\mathcal{F}$
be a coherent sheaf of $\O{X}$-modules. Then there exists a non-empty open subscheme $U\subset S$
such that the restriction of $\mathcal{F}$ to $X_U=p^{-1}(U)$ is flat over $\O{U}$.
\end{thm}
We do not want to go into the details of this statement, as it falls outside of the aim of this paper.
The only things we need to know is, first, that if a ring $A$ is Noetherian, then the structural sheaf
of $\Spec A$ is coherent as a sheaf of modules over itself \cite[5.2.1]{hartshorne}. Moreover, if
$S$ is an affine integral scheme, i.e. $S=\Spec A$ for some domain $A$, then the open subscheme $U$ in Theorem \ref{flatness} is, indeed,
$\Spec A\left[c^{-1}\right]$ for some $c$ coming from local freeness.\par
Now we can begin the actual proof of Theorem \ref{main11}.\par
Here is a diagram including all the maps
involved, so to give also the necessary notation:
\begin{equation*} \xymatrix{
\faktor{K[\un{T},\un{Y}]}{\id{p}_T} \ar[r]^-{\spec_t} & \faktor{K[\un{Y}]}{\id{p}_t} \\
K[\un{T}] \ar@{^{(}->}[u]^-{i_T} \ar[r]^-{\spec_t} & K \ar@{^{(}->}[u]^-{i_t}
}.
\end{equation*}
Here $\spec_t$ is the specialization map at the fixed point $\un{t}\in K^r$. We will show that
both maps $i_\bullet$
are injective.\par
As, by assumption, $\id{p}_T\cap K[\un{T}]=\{0\}$, we have that $i_T$ is an injection.\par
For $i_t$ to be well defined and injective, we show that $\id{p}_t\cap K=\{0\}$, which is
equivalent to showing that
$\id{p}_t\neq K[\un{Y}]$. Consider the ideal $\tilde{\id{p}}_{T}$, which, by Lemma \ref{fractions},
satisfies $\tilde{\id{p}}_{T}\subsetneqq K(\un{T})[\un{Y}]$. By Weak Nullstellensatz \cite[Proposition 9.4.1]{fj},
if $1\notin \tilde{\id{p}}_T\subset K(\un{T})[\un{Y}]$, then there exists  \[\un{x}(\un{T})=(x_1(\un{T}),\ldots,x_s(\un{T}))\in \overline{K(\un{T})}^s\]
such that
\[P_i(\un{T},\un{x}(\un{T}))=0\quad\forall i=1,\ldots,l\text{.}\]\par
For every $\un{t}$ outside of a proper Zariski-closed set $C$ of values, we can extend the morphism of specialization $\spec_t$ to the
$x_i$'s (e.g. \cite[Lemma 1.7.3]{debes}). Then, denoting $\un{x}(\un{t})=(\spec_t(x_1(\un{T})),\ldots,\spec_t(x_s(\un{T})))\in\overline{K}^s$, we have
that
\begin{equation}\label{spec}
P_i(\un{t},\un{x}(\un{t}))=0\quad\forall i=1,\ldots,l    
\end{equation}
which implies that $1\notin\id{p}_t$, so $\id{p}_t\neq K[\un{Y}]$.\par
The above diagram of ring morphisms induces a diagram of scheme morphisms on the spectra of the rings
\[ \xymatrix{
\Spec\left(\faktor{K[\un{T},\un{Y}]}{\id{p}_T}\right) \ar[d]_-{i_T^*}
& \Spec\left(\faktor{K[\un{Y}]}{\id{p}_t}\right) \ar[l]_-{\spec_t^*} \ar[d]_-{i_t^*} \\
\Spec K[\un{T}]   & \Spec K \ar[l]_-{\spec_t^*}
}
\]\par
We look at the map $i_T^*$. 
\begin{itemize}
    \item As $K[\un{T}]$ is a Noetherian domain, $\Spec K[\un{T}]$ is
a Noetherian and integral scheme;
    \item As $\faktor{K[\un{T},\un{Y}]}{\id{p}_T}$ is an algebra of
finite type over $K[\un{T}]$, $i_T^*$ is a morphism of finite type;
    \item Let $\mathcal{F}$ be the structural sheaf of $\Spec K[\un{T}]$. Then $\mathcal{F}$ is coherent on itself.
\end{itemize}

Then we can apply Generic Flatness (Theorem \ref{flatness}): there exists $c(\un{T})\in K[\un{T}]$ such that the following restriction of $i_T^*$
\[i_T^*:\Spec\left(\faktor{K[\un{T},\un{Y}]}{\id{p}_T}\left[c(\un{T})^{-1}\right]\right) \to\Spec \left(K[\un{T}]\left[c(\un{T})^{-1}\right]\right)
\]
is flat. This implies, by \cite[Proposition 9.5, Corollary 9.6]{hartshorne}, that every irreducible
component of $\Spec\left(\faktor{K[\un{T},\un{Y}]}{\id{p}_T}\left[c(\un{T})^{-1}\right]\right)$ has dimension $d$.\par
This yields the following restriction of the initial diagram for every $t\in K^r$ such that $c(\un{T})\neq0$ and $t\notin C$:
\[ \xymatrix{
\Spec\left(\faktor{K[\un{T},\un{Y}]}{\id{p}_T}\left[c(\un{T})^{-1}\right]\right) \ar[d]_-{i_T^*}
& \Spec\left(\faktor{K[\un{Y}]}{\id{p}_t}\right) \ar[l]_-{\spec_t^*} \ar[d]_-{i_t^*} \\
\Spec\left( K[\un{T}]\left[c(\un{T})^{-1}\right]\right)   & \Spec K \ar[l]_-{\spec_t^*}
}
\]
As the dimension of the fiber at a point is preserved by base change, we can conclude that every
irreducible component of $\Spec\faktor{K[\un{Y}]}{\id{p}_t}$ has
dimension $d$ for
every value of $\un{t}\in K^r$ such that $c(\un{t})\neq0$ and $\un{t}\notin C$,
i.e. for every value of $\un{t}\in K^r$ outside of two proper Zariski-closed sets, whose union
is still a Zariski-closed set. Denote this set by $C_1$.
\subsection{Second part}\label{second}
The second stage of the proof is to find $\un{t}$ in $K^r\setminus C_1$ such that the specialized quotient $\faktor{K[\un{Y}]}{\id{p}_t}$ is integral.
This part has a more algebraic approach and relies on the Noether Normalization Lemma. We are stating
below a complete version of this result, coming from the merge of the statements in \cite{hochster} and
\cite[Corollary 13.18]{eisenbud}. It is readily checked that the two proofs can also be merged to yield
the following statement.
\begin{lem}[Noether Normalization Lemma]\label{nn}
Let $A$ be an algebra of finite type of dimension $d$ over a domain $R$. Then there exist a nonzero
element $c\in R$ and elements $z_1,\ldots,z_d$ in $A\left[c^{-1}\right]$, algebraically independent over
$R\left[c^{-1}\right]$, such that $A\left[c^{-1}\right]$ is a module of finite type over its subring
$R\left[c^{-1}][\un{z}\right]\coloneqq R\left[c^{-1}\right][z_1,\ldots,z_d]$.\\
Moreover, set $F=\Frac R$ and $L=\Frac A$. If $L$ is separable over $F$, then $\un{z}$ can be chosen so
to be a separating transcendence basis
of the extension.
\end{lem}
An interesting remark is that the element $c$ satisfying Lemma \ref{freeness} and Theorem \ref{flatness} can also be chosen to satisfy
Lemma \ref{nn}. This is clear by looking at the proofs of these results.\par
Therefore, going back to the proof of Theorem \ref{main2}, we can apply Lemma \ref{nn} to the situation $A=\faktor{K[\un{T},\un{Y}]}{\id{p}_T}$ and $R=K[\un{T}]$. We get that
\[
\faktor{K[\un{T},\un{Y}]}{\id{p}_T}\hspace{-2pt}\left[c(\un{T})^{-1}\right]=K[\un{T}]\hspace{-2pt}\left[c(\un{T})^{-1}\right][\un{z}(\un{T})][\un{\theta}(\un{T})]
\]
for $c(\un{T})\in K[\un{T}]$ the same as in Section \ref{first}, $\un{z}(\un{T})=(z_1(\un{T}),\ldots,z_d(\un{T}))$ a separating transcendence basis in $\faktor{K[\un{T},\un{Y}]}{\id{p}_T}$ and $\un{\theta}(\un{T})=(\theta_1(\un{T}),\ldots,\theta_m(\un{T}))$ the elements
generating $\faktor{K[\un{T},\un{Y}]}{\id{p}_T}\hspace{-2pt}\left[c(\un{T})^{-1}\right]$ as a $K[\un{T}]\hspace{-2pt}\left[c(\un{T})^{-1}\right][\un{z}(\un{T})]$-module. Moreover, $\un{z}(\un{T})$ is
also separating, i.e. the field
$\Frac\left(\faktor{K[\un{T},\un{Y}]}{\id{p}_T}\hspace{-2pt}\left[c(\un{T})^{-1}\right]\right)$ is
algebraically separable over $K(\un{T}, \un{z}(\un{T}))$.\par
Set $R_T\coloneqq K[\un{T}]\hspace{-2pt}\left[c(\un{T})^{-1}\right]$ and $A_T\coloneqq \faktor{K[\un{T},\un{Y}]}{\id{p}_T}\hspace{-2pt}\left[c(\un{T})^{-1}\right] $.
We apply the Primitive Element Theorem as in \cite[Theorem 5.1]{milne}:
there exists an element $\alpha(\un{T})\in\Frac\left(A_T\right) $ such that
\begin{equation}\label{quozienti}
   \Frac\left(A_T\right)
=K(\un{T},\un{z}(\un{T}))(\un{\theta}(\un{T})) =K(\un{T},\un{z}(\un{T}))(\alpha(\un{T}))\text{.}
\end{equation}\par
Moreover, by \cite[Remark 5.2]{milne}, $\alpha(\un{T})$ can be written as a linear combination
\begin{equation}\label{eq1}
 \alpha(\un{T})=\sum_{i=1}^m \alpha_i(\un{T})\theta_i(\un{T})   
\end{equation}
with $\alpha_i(\un{T})\in K[\un{T},\un{z}(\un{T})]$ and chosen to be integral over $R_T[\un{z}(\un{T})]$ (up to multiplying the $\alpha_i(\un{T})$
by some element of $K[\un{T},\un{z}(\un{T})]$).\par
For $i=1,\ldots,m$, let $\delta_i\in R_T[\un{z}(\un{T})]$ such that $\delta_i\theta_i(\un{T})$ is integral over $R_T[\un{z}(\un{T})]$.\par
Let $d(\un{T})\in R_T[\un{z}(\un{T})]$ be the product
of $\delta_1\cdots\delta_m$ with the discriminant of the $K(\un{T},\un{z}(\un{T}))$-basis
\[1,\alpha(\un{T}),\ldots,\alpha(\un{T})^{m-1}\]
of the $m$-dimensional $K(\un{T},\un{z}(\un{T}))$-vector space $K(\un{T},\un{z}(\un{T}))(\alpha(\un{T}))$.\par
As $R_T[\un{z}(\un{T})]$ is integrally closed, it is classical (e.g. \cite[Theorème 1.3.15(a)]{debes}) that
\begin{equation}\label{eq2}
    d(\un{T})\theta_i(\un{T})\in R_T[\un{z}(\un{T})][\alpha(\un{T})]\quad\forall i=1,\ldots,m\text{.}
\end{equation}
Moreover,
our choice of $\alpha(\un{T})$ implies that its minimal polynomial $p(\un{T},\un{z}(\un{T}),Y)$ 
over $K(\un{T},\un{z}(\un{T}))$ is in
$R_T[\un{z}(\un{T}),Y]$.\par
The field $K(\un{T},\un{z}(\un{T}),Y)$ is isomorphic to the field $K(\un{T},\un{W},Y)$ where $\un{W}$
is a new set of variables independent of $\un{T}$. Consider the polynomial $p(\un{T},\un{W},Y)$
image of $p(\un{T},\un{z}(\un{T}),Y)$ via this isomorphism and let 
\[H=\{\un{t}\in K^r|\text{ }p(t,\un{W},Y)\text{ is irreducible in }K[\un{W},Y]\}\]
be the Hilbert set of $p$.\par
For every $\un{t}\in H\subseteq K^r$, the polynomial $p(\un{t},\un{W},Y)$ is irreducible in $K[\un{W},Y]$.\par
It is important to remark that, for every $\un{t}\in K^r\setminus (C_1\cup C_2)$, where $C_2$ is the closed
set defined by $c(\un{t})=0$, a specialization morphism
can be defined that maps $A_T$ to $A_t\hspace{-2pt}\left[c(\un{t})^{-1}\right]$ where $A_t=\faktor{K[\un{Y}]}{\id{p}_t}\left[c(\un{t})^{-1}\right]$. We denote
the images of $\un{z}(\un{T})$ and $\un{\theta}(\un{T})$ via this morphism by
$\un{z}(\un{t})$ and $\un{\theta}(\un{t})$ respectively.\par Furthermore, after specialization in $\un{T}=\un{t}\in K^s\setminus (C_1\cup C_2)$, the elements $\un{z}_i(\un{t})$
are still algebraically independent as Section \ref{first} implies that the transcendence degree is preserved through specialization at $\un{t}$, i.e.
\[d=\trdeg_{K(\un{T})}\Frac(A_T)=\trdeg_K \Frac(A_t)=\trdeg_K K(z_1(\un{t}),\ldots,z_d(\un{t}))\]\par
Therefore, for $\un{t}$ outside of $(C_1\cup C_2)$, $K[\un{z}(\un{t})]$ is still
a polynomial ring of dimension $d$, hence isomorphic to $K[\un{W}]$. As a result, denoting by $\alpha(\un{t})$ the specialization of $\alpha(\un{T})$ given by (\ref{eq1}), the polynomial $p(\un{t},\un{z}(\un{t}),Y)\in K[\un{z}(\un{t}),Y]$ must also be irreducible for $\un{t}\in H\setminus (C_1\cup C_2)$ so
\[K(\un{z}(\un{t}))[\alpha(\un{t})]\cong\faktor{K(\un{z}(\un{t}))[Y]}{\langle p(\un{t},\un{z}(\un{t}),Y)\rangle}\]
is a field.\par
Specializing $\un{T}$ in $\un{t}\in K^r$ outside of the Zariski-closed set $C_3$ defined by $d(\un{t})=0$,
conclusion
(\ref{eq2}) implies that $\theta_i(\un{t})\in K(\un{z}(\un{t}))[\alpha(\un{t})]$ for every $i$.\par
Finally, for $\un{t}\in H\setminus(C_1\cup C_2\cup C_3)$, which is a Hilbert set, $\theta_i(\un{t})\in K(\un{z}(\un{t}))[\alpha(\un{t})]$ for $i=1,\ldots,m$
so $K[\un{z}(\un{t})][\un{\theta}(\un{t})]$ is a subring of $K(\un{z}(\un{t}))[\alpha(\un{t})]$, which is a field, so \[K[\un{z}(\un{t})][\un{\theta}(\un{t})]\cong\faktor{K[\un{Y}]}{\id{p}_t}\]
must be integral. This proves statement (i) of Theorem \ref{main11}.

\section{Theorems \ref{main2} and \ref{main3}}
Before discussing the other two main results, we want to focus on an important tool for their proofs:
{\itshape quasi-generic} polynomials.
\subsection{Quasi-generic polynomials}
In the Introduction, we have briefly talked about generic polynomials. In fact, we want to define a
larger class of polynomials, the {\itshape quasi-generic polynomials}, of which the generic
polynomial is the principal example.
\begin{defin}\label{qg}
Let $K$ be a field, $K[\un{Y}]$ the ring of polynomials with coefficients in $K$ and variables $\un{Y}$.
Given an integer $D\geq0$, a set \[S=\{Q_1(\un{Y}),\ldots,Q_{|S|}(\un{Y})\}\subseteq\{Y_1^{\beta_1}\cdots Y_s^{\beta_s},\beta_i\geq0\text{ and }\sum_{i=1}^s\beta_i\leq D\}\] of power products of degree at most $D$, which always contains $Q_1(\un{Y})=1$ and a polynomial $R(\un{Y})\in K[\un{Y}]$, we define the {\itshape quasi-generic}
polynomial of base $S,R$:
\[\mathcal{Q}_{S,R}(\un{\Lambda},\un{Y})=\sum_{i=1}^{|S|} \Lambda_i\hspace{1pt}Q_i(\un{Y}) +R(\un{Y})\]
where $\un{\Lambda}=(\Lambda_1,\ldots,\Lambda_{|S|})$ is a new set of variables called the set of parameters.
\end{defin}
We note that, by taking all
the power products for $i=1,\ldots,|S|$ and $R(\un{Y})=0$, we obtain the generic polynomial of degree $D$.\par
The importance of such polynomials is shown in the following lemma.
\begin{lem}\label{ramero}
Let $K$ be a field. Let $\id{p}$ be a prime ideal in $K[\un{Y}]$ of height $\height \id{p}$ and $\mathcal{Q}_{S,R}(\un{\Lambda},\un{Y})$
a quasi-generic polynomial. Assume that $S$, $\id{p}$ and $R(\un{Y})$ satisfy hypothesis {\normalfont(\ref{H})} stated below.
Denote by $\id{P}$ the ideal $\langle\id{p},\mathcal{Q}_{S,R}\rangle\subseteq K[\un{\Lambda},\un{Y}]$
and by $\tilde{\id{P}}$ the ideal $\langle\id{p},\mathcal{Q}_{S,R}\rangle\subseteq K(\un{\Lambda})[\un{Y}]$.
Then $\tilde{\id{P}}$ is a prime ideal of height $\height\tilde{\id{P}}=\height\id{p}+1$.
\end{lem}
To state hypothesis (\ref{H}), consider
the set $E$ of elements in $B\coloneqq\faktor{K[\un{Y}]}{\id{p}}$ which are algebraic over $K$.
Clearly $E$ is a field containing $K$.
Let then
\[\phi_{S,R}:K^{|S|-1}\to \faktor{B}{E}\]
be the map sending an $(|S|-1)$-uple $(a_2,\ldots,a_{|S|})$ to the coset modulo $E$ of the element $\sum_{i=2}^{|S|} a_iQ_i(\un{Y})+R(\un{Y})$.
\begin{defin}
The triple $(\id{p},S,R(\un{Y}))$ satisfies hypothesis {\normalfont(\ref{H})} if
\leqnomode
\begin{equation}\tag{H}\label{H}
    \phi_{S,R}\text{ is not identically zero.}
\end{equation}
\reqnomode
\end{defin}
\begin{lem}\label{lemh}
{\normalfont (i)} If the triple $(\id{p},S,R(\un{Y}))$ satisfies hypothesis {\normalfont(\ref{H})}, then $\id{p}$
is a non-maximal ideal of $K[\un{Y}]$.\par
{\normalfont (ii)} If $\id{p}$ is a non-maximal of $K[\un{Y}]$ and $\{Y_1,\ldots,Y_s\}\subset S\cup\{R(\un{Y})\}$, then the triple $(\id{p},S,R(\un{Y}))$ satisfies hypothesis {\normalfont(\ref{H})}.
\end{lem}
\begin{proof}
(i) By contradiction, assume that $\id{p}$ is maximal. Then $B$ is a $K$-algebra of finite type and
a field, so by \cite[Corollary 5.24]{am} $B$ is an algebraic extension of $K$, hence $B=E$ and $\phi_{S,R}$
is identically zero.\par
(ii) By contradiction, assume that $\phi_{S,R}$ is identically zero. Then, $\phi_{S,R}(0,\ldots,0)=0$
and $\phi_{S,R}(e_i)=0$ for every $i$, where $\{e_i,\text{ }i=1,\ldots,|S|-1\}$ is the canonical
base of $K^{|S|-1}$ as a $K$-vector space. Then $Y_i\in E$ for every $i$, but
$\{Y_i,\text{ }i=1,\ldots,s\}$ generates $B$ over $K$, hence $B=E$, i.e. $B$ is a field, which is a
contradiction with $\id{p}$
being non-maximal.
\end{proof}
\begin{proof}[Proof of Lemma \ref{ramero}]
{\bfseries First step.} We show that $\id{P}$ is a prime ideal of $K[\un{\Lambda},\un{Y}]$.\par
Using a similar strategy as in \cite[Lemma 2.1(a)]{bdn}, consider the ring automorphism
\begin{equation}\label{auto}
    f:K[\un{\Lambda},\un{Y}]\to K[\un{\Lambda},\un{Y}]
\end{equation}which is the identity on
$K[\Lambda_2,\ldots,\Lambda_{|S|},\un{Y}]$ and sends $\Lambda_1$ to $\Lambda_1-\sum_{i=2}^{|S|} \Lambda_i Q_i(\un{Y})-R(\un{Y})$. The
ideal $\langle\id{p},\mathcal{Q}_{S,R}\rangle$ is then sent to the ideal $\langle\id{p},\Lambda_1\rangle$.\par
Now consider the specialization morphism, $f_0:K[\un{\Lambda},\un{Y}]\to K[\Lambda_2,\ldots,\Lambda_{|S|},\un{Y}]$ sending
$\Lambda_1$ to $0$. The ideal $\langle\id{p}\rangle$ in $K[\Lambda_2,\ldots,\Lambda_{|S|},\un{Y}]$ is prime as the following isomorphism shows 
\begin{equation}\label{ext}
\faktor{K[\Lambda_2,\ldots,\Lambda_{|S|},\un{Y}]}{\langle\id{p}\rangle}\cong\faktor{K[\un{Y}]}{\id{p}}[\Lambda_2,\ldots,\Lambda_{|S|}]\text{.}
\end{equation}
So its preimage under $f_0$, i.e. the ideal $\id{p}+\ker f_0=\langle\id{p},\Lambda_1\rangle$ is also prime.\par
As a result, the ideal $\id{P}=\langle\id{p},\mathcal{Q}_{S,R}\rangle$ is prime in $K[\un{\Lambda},\un{Y}]$, being sent to a prime ideal by $f$.\par\smallskip
{\bfseries Second step.} We show that $\mathcal{Q}_{S,R}$ is not invertible in the ring $B_{\Lambda}\coloneqq \faktor{K(\un{\Lambda})[\un{Y}]}{\tilde{\id{p}}_{\Lambda}}$, where $\tilde{\id{p}}_\Lambda$ is the extension of $\id{p}$ to $K(\un{\Lambda})[\un{Y}]$.\par
We note that the quotient $B_{\Lambda}$ is integral and non-trivial. Indeed, the ideal
$\tilde{\id{p}}_{\Lambda}$ is prime in $K(\un{\Lambda})[\un{Y}]$: the ideal
$\id{p}_{\Lambda}=\langle\id{p}\rangle\subset K[\un{\Lambda},\un{Y}]$ is prime (proceed similarly
as in (\ref{ext})) and $\id{p}_{\Lambda}\cap K[\un{\Lambda}]=\{0\}$ because, otherwise, if there was
some nonzero
$P(\un{\Lambda})$ in $\id{p}_{\Lambda}$, then for every $\un{\lambda}\in K^{|S|}$ such that $P(\un{\lambda})\neq0$,
$P(\un{\lambda})\in\id{p}$, which is a contradiction because $\id{p}\neq  K[\un{Y}]$.\par
Now, by contradiction, assume that $\mathcal{Q}_{S,R}$ is invertible in $B_{\Lambda}$.\par
Then, there exists $\alpha\in B_{\Lambda}$ such that $\alpha\hspace{1pt} \mathcal{Q}_{S,R}=1$.
As $B_{\Lambda}=S^{-1}B[\un{\Lambda}]$ with $S=K[\un{\Lambda}]$, we can write $\alpha=\frac{N(\un{\Lambda})}{P(\un{\Lambda})}$ for $N\in B[\un{\Lambda}]$ and
$P\in K[\un{\Lambda}]$, $P\neq0$.\par
As, by hypothesis (\ref{H}), $\phi_{S,R}$ is not identically $0$, the linear subvariety $V=\phi_{S,R}^{-1}(0)$ is of dimension strictly smaller
that $|S|-1$.\par Define the set
\[Z\coloneqq\{\un{a}=(a_2,\ldots,a_r)\in K^{|S|-1} : P(\Lambda_1,a_2,\ldots,a_{|S|})=0\}\text{.}\]
If we write $P(\un{\Lambda})=\sum _{i=1}^k p_i(\Lambda_2,\ldots,\Lambda_{|S|})\Lambda_1^i$, then we see that
$Z=\bigcap_{i=0}^k V(p_i)$, where $V(p_i)$
is the zero locus of $p_i$ in $K^{|S|-1}$. We distinguish two cases.\par
{\itshape First case.} Assume that $K$ is infinite.\par
The polynomial $P(\un{\Lambda})$ is nonzero, so, in particular, there exists $i$ such that $p_i\neq0$.
As $Z\subseteq V(p_i)$, then $Z\cup V\subset V(p_i)\cup V$. The set $V(p_i)\cup V$ is a proper closed set because
union of two proper closed sets. Thus, if $K$ is infinite, $V(p_i)\cup V\neq K^{|S|-1}$, hence $V\cup Z\neq K^{|S|-1}$.\par
Take then $\un{a}\in K^{|S|-1}\setminus (V\cup Z)$.
Recall that $N(\un{\Lambda})\mathcal{Q}_{S,R}(\un{\Lambda},\un{Y}) =P(\un{\Lambda})$. So, as $\mathcal{Q}_{S,R}(\un{\Lambda},\un{Y})$ divides $P(\un{\Lambda})$ in
$B[\un{\Lambda}]$, it follows that $m(\Lambda_1)\coloneqq \mathcal{Q}_{S,R}(\Lambda_1,\un{a},\un{Y})$
divides $p(\Lambda_1)\coloneqq P(\Lambda_1,\un{a})$ in $B[\Lambda_1]$.\par
By construction of $\un{a}$, we have $p(\Lambda_1)\neq0$ and $m(\Lambda_1)=\Lambda_1+Q(\un{Y})$ with $Q(\un{Y})=\sum_{i=2}^{|S|} a_iQ_i(\un{Y})+R(\un{Y})$. Then $\Lambda_1=-Q(\un{Y})$ is a root of $p(\Lambda_1)=0$ which, by construction, has coefficients in $K$ so its roots are algebraic over $K$.
But $Q(\un{Y})$ is transcendental over $K$: the coset modulo $E$ of $Q(\un{Y})$
is $\phi_{S,R}(\un{a})\neq0$ because $\un{a}\notin V$, so $Q(\un{Y})\notin E$. This is a contradiction.\par
{\itshape Second case.} Assume that $K$ is finite. Let $K'$ be an algebraic closure of $K$. By 
\cite[Theorem 5.10]{am}, there exists a prime ideal $\id{p}'$ in $K'[\un{Y}]$ such that $\id{p}'\cap K[\un{Y}]=\id{p}$. Moreover, by the Going-up Theorem \cite[Theorem 5.11]{am} and the incomparability
property \cite[Corollary 5.9]{am} we have $\height\id{p}'=\height\id{p}$.\par
Replacing $K$ and $\id{p}$ by
$K'$ and $\id{p}'$ we can get back to the first case. Indeed, define $B'\coloneqq\faktor{K'[\un{Y}]}{\id{p}'}$ and $B'_{\Lambda}\coloneqq K(\un{\Lambda})\otimes_{K[\un{\Lambda}]}B'[\un{\Lambda}]$
and apply the first case to the image of $\mathcal{Q}_{S,R}$ under the induced homomorphism
$B_{\Lambda}\to B'_{\Lambda}$. The image of the polynomial $\mathcal{Q}_{S,R}$ is then not invertible
in $B'_{\Lambda}$, which implies that $\mathcal{Q}_{S,R}$ is not invertible in $B_{\Lambda}$, for
otherwise the previous homomorphism yields an invertible element in $B'_{\Lambda}$. \par\smallskip
{\bfseries Third step.} The fact that $\mathcal{Q}_{S,R}$ is not invertible in the ring $B_{\Lambda}$ implies that $\id{P}\cap K[\un{\Lambda}]=\{0\}$. If this was not the case, we would have $\tilde{\id{P}}=K(\un{\Lambda})[\un{Y}]$.
But, then, we could find $A(\un{\Lambda},\un{Y}), B(\un{\Lambda},\un{Y})\in K(\un{\Lambda})[\un{Y}]$ and
$P(\un{Y})\in\id{p}$ such that
\[A(\un{\Lambda},\un{Y})P(\un{Y})+B(\un{\Lambda},\un{Y})\mathcal{Q}_{S,R}(\un{\Lambda},\un{Y})=1\text{.}\]
Reducing this equality modulo $\tilde{\id{p}}_{\Lambda}$, we would obtain that $\mathcal{Q}_{S,R}$ is invertible in $B_{\Lambda}$, which is a contradiction.\par
Saying that that $\id{P}\cap K[\un{\Lambda}]=\{0\}$ is also equivalent to saying that $\tilde{\id{P}}$ is a prime ideal of $K(\un{\Lambda})[\un{Y}]$, by bijective correspondence \cite[Proposition 3.11(iv)]{am}.\par\smallskip
{\bfseries Fourth step.} The polynomial $\mathcal{Q}_{S,R}$ is not contained in $\tilde{\id{p}}_{\Lambda}$, i.e.
$\tilde{\id{p}}_{\Lambda}\subsetneqq\tilde{\id{P}}$. Otherwise, we could write the
following relation
\[\mathcal{Q}_{S,R}(\un{\Lambda},\un{Y})=\sum_{i=1}^n A_i(\un{\Lambda},\un{Y})P_i(\un{Y})\]
for $A_i\in K(\un{\Lambda})[\un{Y}]$ and $P_i(\un{Y})\in\id{p}$. Specializing this equality in $\un{\lambda}=\un{0}$ and $\un{\lambda}=(1,0,\ldots,0)$, we would find that $R$ and $1+R$, respectively, belong to $\id{p}$, so $1\in\id{p}$, which is a contradiction.\par
\smallskip
{\bfseries Fifth step.} It follows from $\tilde{\id{P}}$ being a prime ideal and $\tilde{\id{p}}_{\Lambda}\nsubseteq\tilde{\id{P}}$ that the quotient
$\faktor{\tilde{\id{P}}}{\tilde{\id{p}}_{\Lambda}}$ is a nonzero prime ideal
of $B_{\Lambda}$.\par
The ring $B_{\Lambda}$ is integral and Noetherian by construction and the element
$\mathcal{Q}_{S,R}\mod \tilde{\id{p}}_{\Lambda}$ is nonzero and is not invertible in $B_{\Lambda}$. By Krull's Height Theorem
\cite[Theorem 1.11A]{hartshorne}, the ideal $\faktor{\tilde{\id{P}}}{\tilde{\id{p}}_{\Lambda}}\subset B_{\Lambda}$ has height $1$, so the ideal $\tilde{\id{P}}$ has height
$\height\tilde{\id{p}}_{\Lambda}+1=\height\id{p}+1$ in $K(\un{\Lambda})[\un{Y}]$.
\end{proof}
Now, consider the ideal $\id{P}=\langle\id{p},\mathcal{Q}_{S,R}\rangle\subset K[\un{\Lambda},\un{Y}]$. If we assume that $K$ has characteristic $0$, we have just showed that
$\id{P}$ satisfies all the hypotheses of Theorem \ref{main11}. Its conclusion already proves
Theorem \ref{main2} as it is stated in the Introduction. As promised we will establish a more general version
using several quasi-generic polynomials.\par 
In the following
sections we present
two recursive generalizations of Lemma \ref{ramero} and see how they imply Theorem \ref{main2} (generalized) and Theorem \ref{main3}.

\subsection{Intersection of varieties} Fix $\rho>0$. For $i=1,\ldots,\rho$, fix a non-negative
integer $D_i$ and then consider the quasi-generic polynomial $\mathcal{Q}_{S_i,R_i}(\un{\Lambda}_i,\un{Y})$
of basis the set $S_i$ of all the power products in the variables $\un{Y}$ of degree
$\leq D_i$ and $R_i=0$; the additional variables $\un{\Lambda}_i$ form the \say{set of parameters}
of Definition \ref{qg}. In fact, given this choice of
$S_i$ and $R_i$, the polynomial $\mathcal{Q}_{S_i,R_i}(\un{\Lambda}_i,\un{Y})$ is the generic polynomial of degree
$D_i$, so, in this section, we will denote it by $\mathcal{Q}_{D_i}(\un{\Lambda}_i,\un{Y})$.\par
Set $K[\un{\un{\Lambda}},\un{Y}]=K[\un{\Lambda}_1,\ldots,\un{\Lambda}_{\rho},\un{Y}]$ where
$\un{\un{\Lambda}}=(\un{\Lambda}_1,\ldots,\un{\Lambda}_{\rho})$.\par
The following statement generalizes Lemma \ref{ramero} for this set of data.
\begin{thm}\label{ramero1}
Let $K$ be a field. Let $\id{p}$ be a non-maximal prime ideal of $K[\un{Y}]$ such that
$\dim_K\left(\faktor{K[\un{Y}]}{\id{p}}\right)=d>0$. Let $\mathcal{Q}_{D_i}(\un{\Lambda}_1,\un{Y}),\ldots,\mathcal{Q}_{D_{\rho}}(\un{\Lambda}_{\rho},\un{Y})$ be the generic polynomials defined above for $0<\rho\leq d$. Then the ideal $\id{P}_{\rho}=
\langle\id{p},\mathcal{Q}_{D_1},\ldots,\mathcal{Q}_{D_{\rho}}\rangle$ is a prime ideal of $K[\un{\un{\Lambda}},\un{Y}]$ such that
$\id{P}_{\rho}\cap K[\un{\un{\Lambda}}]=\{0\}$ and
\[\dim_{K(\un{\un{\Lambda}})}\left(\faktor{K(\un{\un{\Lambda}})[\un{Y}]}{\tilde{\id{P}}_{\rho}}\right)=d-\rho,\] where $\tilde{\id{P}}_{\rho}$ is the extension of $\id{P}_{\rho}$ to $K(\un{\un{\Lambda}})[\un{Y}]$.
\end{thm}
\begin{proof}
We proceed by recursion on $\rho$.\par
The case $\boldsymbol{\rho=1}$ is exactly Lemma \ref{ramero} where $\id{P}_1$ is the ideal $\id{P}$ in the
statement of the lemma and, consequently, $\tilde{\id{P}}_1$ is the
ideal $\tilde{\id{P}}$. As previously remarked, the fact that $\height\tilde{\id{P}}_1=\height\id{p}+1$ is equivalent to saying that \[\dim_{K(\un{\un{\Lambda}})}\left(\faktor{K(\un{\un{\Lambda}})[\un{Y}]}{\tilde{\id{P}}_1}\right)=
\dim_K\left(\faktor{K[\un{Y}]}{\id{p}}\right)-1=d-1\text{.}\]\par\smallskip
For simplicity in the notation, we only explain the case $\boldsymbol{\rho=2}$. It will then be clear how
to prove the case for an arbitrary $\rho\leq d$.\par
Let $\id{P}_1=\langle\id{p},\mathcal{Q}_{D_1}\rangle\subset K[\un{\Lambda}_1,\un{Y}]$ be the ideal
obtained as in the case $\rho=1$. As $\dim K[\un{\Lambda}_1,\un{Y}]>\dim K[\un{Y}]$ and
$\height\id{P}_1=\height\id{p}+1$, the ideal $\id{P}_1$ is not maximal.
Moreover, by Lemma \ref{lemh}(ii), the triple $(\id{P}_1, S_2, 0)$ satisfies hypothesis (\ref{H}) because $\mathcal{Q}_{D_2}$
is the generic polynomial of degree $D_2$. Therefore, we can apply Lemma \ref{ramero}
to $\id{P}_1$ and $\mathcal{Q}_{D_2}$ and obtain that $\id{P}_2=\langle\id{p},\mathcal{Q}_{D_1},\mathcal{Q}_{D_2}\rangle$ is prime in $K[\un{\Lambda}_1,\un{\Lambda}_2,\un{Y}]$ and has height $\height \id{P}_2=\height\id{p}+2$,
i.e.
\[\dim_{K(\un{\un{\Lambda}})}\left(\faktor{K(\un{\un{\Lambda}})[\un{Y}]}{\tilde{\id{P}}_2}\right)=
\dim_K\left(\faktor{K[\un{Y}]}{\id{p}}\right)-2=d-2\text{.}\]
\end{proof}
Denote by $V_{\rho,\un{\un{\Lambda}}}$ the variety defined by $\mathcal{Q}_{D_1}(\un{\Lambda}_1,\un{Y}),\ldots,\mathcal{Q}_{D_{\rho}}(\un{\Lambda}_{\rho},\un{Y})$.
If $0<\rho\leq s$, as it is the case if $0<\rho\leq d$ as above, a recursive application of Lemma
\ref{ramero}, starting with $\id{p}=\langle\mathcal{Q}_{D_1}(\un{\Lambda}_1,\un{Y})\rangle$, easily shows
that $V_{\rho,\un{\un{\Lambda}}}$ is, in fact, an irreducible $K(\un{\un{\Lambda}})$-variety
of codimension $\rho$, i.e. $\langle\mathcal{Q}_{D_1}(\un{\Lambda}_1,\un{Y}),\ldots,\mathcal{Q}_{D_{\rho}}(\un{\Lambda}_{\rho},\un{Y})\rangle$ is a prime ideal of height $\rho$ in $K[\un{\un{\Lambda}},\un{Y}]$. We call
$V_{\rho,\un{\un{\Lambda}}}$ the {\itshape generic $K(\un{\un{\Lambda}})$-subvariety of codimension $\rho$}.\par
Using this remark, a general version of Theorem \ref{main2} follows from conjoining Theorem \ref{ramero1}
and Theorem \ref{main11}.
\begin{cor}\label{main21}
Let $K$ be a field of characteristic $0$. Let $V=V_K(\id{p})$ be an irreducible $K$-variety such that
$\dim_K\left(\faktor{K[\un{Y}]}{\id{p}}\right)=d>0$. Let $V_{\rho,\un{\un{\Lambda}}}$ be the generic
$K(\un{\un{\Lambda}})$-subvariety defined above. Then for
$\un{\un{\lambda}}=(\un{\lambda}_1,\ldots,\un{\lambda}_{\rho})$ in some Hilbert subset of
$K^{N_{D_1}+\ldots+N_{D_{\rho}}}$, the intersection $V\cap V_{\rho,\un{\un{\lambda}}}$ of $V$ with the
$K$-variety $V_{\rho,\un{\un{\lambda}}}$, obtained by specializing $\un{\un{\Lambda}}$ at $\un{\un{\lambda}}$, is an irreducible $K$-variety of dimension $d-\rho$.
\end{cor}
Theorem \ref{main2} is the special case for which $\rho=1$ and $\mathcal{Q}_{S,R}(\un{\Lambda},\un{Y})$ is the
generic polynomial of degree $D$.
\begin{proof}
By Theorem \ref{ramero1}, the ideal $\id{P}_{\rho}=\langle\id{p},\mathcal{Q}_{D_1},\ldots,\mathcal{Q}_{D_{\rho}}\rangle$ is prime in $K[\un{\un{\Lambda}},\un{Y}]$ and $\id{P}_{\rho}\cap K[\un{\un{\Lambda}}]=\{0\}$.
Moreover, as $K$ has characteristic $0$,
$\Frac\left(\faktor{K[\un{\un{\Lambda}},\un{Y}]}{\id{P}_{\un{S}}}\right)$ is separable over $K(\un{\un{\Lambda}},\un{Y})$. Then we can apply Theorem \ref{main11} to $\id{P}_{\rho}$:
using statement (iii) of the theorem,
for $\un{\un{\lambda}}=(\un{\lambda}_1,\ldots,\un{\lambda}_{\rho})$ in a Hilbert subset of $K^{N_{D_1}+\ldots+N_{D_{\rho}}}$, the $K$-variety
\[V_K(\id{p},\mathcal{Q}_{D_1}(\un{\lambda}_1,\un{Y}),\ldots,\mathcal{Q}_{D_{\rho}}(\un{\lambda}_{\rho},\un{Y}))=V\cap V_{\rho,\un{\un{\lambda}}}\]
is an irreducible $K$-variety of dimension $d-\rho$.
\end{proof}
\begin{oss}
At the beginning of the section, we chose to take as $\mathcal{Q}_{S_i,R_i}(\un{\Lambda}_i,\un{Y})$ the
generic polynomial of degree $D_i$. However, if we fix the ideal $\id{p}$ at the beginning, Corollary \ref{main21} holds more generally if we take for $S_i$ a \un{subset} of all possible monomials such that
the triple $(\id{p},S_i,0)$ satisfies hypothesis (\ref{H}) and Theorem \ref{ramero1}.
\end{oss}

\subsection{Specialization at polynomials}
In the previous sections the surrounding ring used to define the quasi-generic polynomials
was $K[\un{Y}]$, while in this section it will be $K[\un{T},\un{Y}]$.\par
Fix $\rho>0$. For $i=1,\ldots,\rho$, fix a non-negative
integer $D_i$ and then consider the quasi-generic polynomial $\mathcal{Q}_{S_i,R_i}(\un{\Lambda}_i,\un{Y})$
of basis the set $S_i$ of all the power products in the variables $\un{Y}$ of degree $\leq D_i$ and $R_i=-T_i$; the additional variables $\un{\Lambda}_i$ form the \say{set of parameters} of Definition \ref{qg}. Thus, we have
\[\mathcal{Q}_{S_i,R_i}(\un{\Lambda}_i,\un{T},\un{Y})=\sum_{j=1}^{N_{D_i}}\Lambda_{i,j}\hspace{1pt}Q_j(\un{Y})-T_i
=\mathcal{U}_{D_i}(\un{\Lambda}_i,\un{Y})-T_i\text{.}\]
Note that $\mathcal{U}_{D_i}(\un{\Lambda}_i,\un{Y})$, as defined above, is the generic polynomial of degree $D_i$ in the variables $\un{Y}$.\par
According to the definition of quasi-generic polynomial, the power products could be taken in the variables
$\un{T}$ and $\un{Y}$, but we take them only in the variables $\un{Y}$ for our
purpose.\par
The following statement generalizes Lemma \ref{ramero} for this set of data.
\begin{thm}\label{ramero2}
Let $K$ be a field. Let $\id{p}$ be a prime ideal of $K[\un{T},\un{Y}]$ such that $\id{p}\cap K[\un{T}]=\{0\}$ and
$\dim_{K(\un{T})}\left(\faktor{K(\un{T})[\un{Y}]}{\tilde{\id{p}}}\right)=d>0$, where $\tilde{\id{p}}$
is the extension of $\id{p}$ to $K(\un{T})[\un{Y}]$. Let $\mathcal{Q}_{S_1,R_1}(\un{\Lambda}_1,\un{T},\un{Y}),\ldots,\mathcal{Q}_{S_{\rho},R_{\rho}}(\un{\Lambda}_{\rho},\un{T},\un{Y})$ be the quasi-generic polynomials defined above for $0<\rho\leq r$. 
Then the ideal $\id{P}_{\un{S}}=
\langle\id{p},\mathcal{Q}_{S_1,R_1},\ldots,\mathcal{Q}_{S_{\rho},R_{\rho}}\rangle$ is a prime ideal of $K[\un{\un{\Lambda}},\un{T},\un{Y}]$ such that
$\id{P}_{\un{S}}\cap K[\un{\un{\Lambda}}]=\{0\}$ and
\[\dim_{K(\un{\un{\Lambda}})}\left(\faktor{K(\un{\un{\Lambda}})[\un{T},\un{Y}]}{\tilde{\id{P}}_{\un{S}}}\right)=d+r-\rho\]
where $\tilde{\id{P}}_{\un{S}}$ is the extension of $\id{P}_{\un{S}}$ to $K(\un{\un{\Lambda}})[\un{T},\un{Y}]$.
\end{thm}
\begin{proof}
We proceed by recursion on $\rho$. \par Assume $\boldsymbol{\rho=1}$. As $\id{p}\cap K[\un{T}]=\{0\}$, in particular, $\id{p}\cap K[T_2,\ldots,T_r]=\{0\}$,
so the ideal $\tilde{\id{p}}_{T_2}=\langle\id{p}\rangle\subset K(T_2,\ldots,T_r)[T_1,\un{Y}]$ is non-maximal by Lemma \ref{fractions}.\par
By construction, the set $S_1$ contains all the power products in the variables $\un{Y}$ of degree $\leq D_1$,
hence $Y_j$, for all $j=1,\ldots,s$, and we have set $R_1(\un{T},\un{Y})=-T_1$. So, by Lemma \ref{lemh}(ii),
the triple $(\tilde{\id{p}}_{T_2}, S_1,-T_1)$ satisfies
hypothesis (\ref{H}) with the ring $K[\un{Y}]$ in Lemma \ref{lemh} replaced by
$K(T_2,\ldots,T_r)[T_1,\un{Y}]$.\par
Therefore, by Lemma \ref{ramero}, the ideal \[\tilde{\id{P}}_{T_2,S_1}=\langle\id{p},\mathcal{Q}_{S_1,R_1}\rangle\subset K(\un{\Lambda}_1,T_2,\ldots,T_r)[T_1,\un{Y}]\]
is a prime ideal and $\height\tilde{\id{P}}_{T_2,S_1}=\height\id{p}+1$.\par
By the classical bijective correspondence between extended and contracted ideals in rings of fractions
(e.g. \cite[Proposition 3.11(iv)]{am}), to the ideal $\tilde{\id{P}}_{T_2,S_1}$ we associate the prime ideal
\[\tilde{\id{P}}_{S_1}=\langle\id{p},\mathcal{Q}_{S_1,R_1}\rangle\subset 
K(\un{\Lambda}_1)[\un{T},\un{Y}].\] In the same manner, we associate the prime ideal
\[\id{P}_{S_1}=\langle\id{p},\mathcal{Q}_{S_1,R_1}\rangle
\subset K[\un{\Lambda}_1,\un{T},\un{Y}]\] and, in addition, we have $\id{P}_{S_1}\cap K[\un{\Lambda}_1]=\{0\}$.\par
Moreover,
by Lemma \ref{fractions}, $\height \tilde{\id{P}}_{S_1}=\height\tilde{\id{P}}_{T_2,S_1}=\height \id{p}+1$,
so \[\dim_{K(\un{\Lambda}_1)}\left(\faktor{K(\un{\Lambda}_1)[\un{T},\un{Y}]}{\tilde{\id{P}}_{S_1}}\right)=
r+s-(\height\id{p}+1)=d+r-1\text{.}\]\par\smallskip
For simplicity in the notation, we explain only the case $\boldsymbol{\rho=2}$. The case of an arbitrary $\rho\leq r$
can be easily deduced.\par
Consider the prime ideal \[\tilde{\id{P}}_{T_3,S_1}=\langle\id{p},\mathcal{Q}_{S_1,R_1}\rangle
\subset K(\un{\Lambda}_1,T_3,\ldots,T_r)[T_1,T_2,\un{Y}]\] deduced from $\id{P}_{S_1}$ by
applying the classical bijective correspondence. Denote by \[\tilde{\id{P}}_{T_1^*,T_3} \subset K(\un{\Lambda}_1,T_3,\ldots,T_r)[T_2,\un{Y}]\]
the ideal obtained by replacing $T_1$ with the generic polynomial, previously denoted by $\mathcal{U}_{D_1}(\un{\Lambda}_1,\un{Y})$, in
$\tilde{\id{P}}_{T_3,S_1}$. For $\rho=2$, the ideal $\tilde{\id{P}}_{T_1^*,T_3}$ will play the role
played by $\tilde{\id{p}}_{T_2}$ in the case $\rho=1$.\par
The ideal $\tilde{\id{P}}_{T_1^*,T_3}$
is formally constructed through the quotient morphism, which we denote by $\pi_1$, sending $\tilde{\id{P}}_{T_3,S_1}$ to
\begin{equation*}
\faktor{\tilde{\id{P}}_{T_3,S_1}}{\langle\mathcal{Q}_{S_1,R_1}\rangle}\cong\tilde{\id{P}}_{T_1^*,T_3}.
\footnotemark
\end{equation*}
\footnotetext{Recall that $\mathcal{Q}_{S_1,R_1}=\mathcal{U}_{D_1}(\un{\Lambda}_1,\un{Y})-T_1$}
It follows that $\tilde{\id{P}}_{T_1^*,T_3}$ is prime. Moreover, by \cite[Proposition 9.2]{eisenbud},
we have
$\height\tilde{\id{P}}_{T_1^*,T_3}=\height\tilde{\id{P}}_{T_3,S_1}-1$. Now, by Lemma \ref{fractions} and
the case $\rho=1$, we have
$\height\tilde{\id{P}}_{T_3,S_1}=\height\id{P}_{S_1}=\height \id{p}+1$. Therefore, we obtain:
\begin{equation}\label{height}
    \height\tilde{\id{P}}_{T_1^*,T_3}=\height \id{p},
\end{equation}
so $\tilde{\id{P}}_{T_1^*,T_3}$ is a non-maximal prime ideal of $K(\un{\Lambda}_1,T_3,\ldots,T_r)[T_2,\un{Y}]$.\par
Consider the polynomial $\mathcal{Q}_{S_2,R_2}(\un{\Lambda}_2,\un{T},\un{Y})$. By construction,
$S_2$ contains $Y_j$ for all $j=1,\ldots,s$ and $R_2(\un{T},\un{Y})=-T_2$. By Lemma \ref{lemh}(ii), 
the triple $(\tilde{\id{P}}_{T_1^*,T_3},S_2,-T_2)$
satisfies hypothesis (\ref{H}) with the ring $K[\un{Y}]$ in Lemma \ref{lemh} replaced by the ring $K(\un{\Lambda}_1,T_3,\ldots,T_r)[T_2,\un{Y}]$.\par
From Lemma \ref{ramero} applied to $\tilde{\id{P}}_{T_1^*,T_3}$ and $\mathcal{Q}_{S_2,R_2}$, we deduce that
the ideal \[\tilde{\id{P}}_{T_1^*,T_3,S_2}\coloneqq\langle \tilde{\id{P}}_{T_1^*,T_3},\mathcal{Q}_{S_2,R_2}\rangle\subset K(\un{\Lambda}_1,\un{\Lambda}_2,T_3,\ldots,T_r)[T_2,\un{Y}]\] is
prime and has height
$\height\tilde{\id{P}}_{T_1^*,T_3,S_2}=\height\id{p}+1$.\par
Using the morphism $\pi_1$, we obtain that the ideal
\[\tilde{\id{P}}_{T_1^*,T_3,S_2}+\ker\pi_1=\langle \tilde{\id{P}}_{T_3,S_1},\mathcal{Q}_{S_2,R_2}\rangle=\langle\id{p},\mathcal{Q}_{S_1,R_1},\mathcal{Q}_{S_2,R_2}\rangle\]
is a prime ideal of $K(\un{\Lambda}_1,\un{\Lambda}_2,T_3,\ldots,T_r)[T_1,T_2,\un{Y}]$
and that its height is equal to
\[\height\tilde{\id{P}}_{T_1^*,T_3,S_2}+1=\height\id{p}+2.\]\par
Applying the classical bijective correspondence, the ideal
\[\id{P}_{S_1,S_2}=\langle\id{p},\mathcal{Q}_{S_1,R_1},\mathcal{Q}_{S_2,R_2}\rangle\subset K[\un{\Lambda}_1,\un{\Lambda}_2,\un{T},\un{Y}] \] is prime and such that $\id{P}_{S_1,S_2}\cap
K[\un{\Lambda}_1,\un{\Lambda}_2]=\{0\}$. Moreover, the ideal
\[\tilde{\id{P}}_{S_1,S_2}=\langle\id{p},\mathcal{Q}_{S_1,R_1},\mathcal{Q}_{S_2,R_2}\rangle\subset K(\un{\Lambda}_1,\un{\Lambda}_2)[\un{T},\un{Y}]\] is prime. By Lemma \ref{fractions},
both these ideals have height $\height\id{p}+2$.\par
In terms of dimensions, this is equivalent to saying that
\[\dim_{K(\un{\Lambda}_1,\un{\Lambda}_2)}\left(\faktor{K(\un{\Lambda}_1,\un{\Lambda}_2)[\un{T},\un{Y}]}{\tilde{\id{P}}_{S_1,S_2}}\right)=d+r-2\text{.}\]
\end{proof}
Fix $\rho=r$. Theorem \ref{main3} follows from Theorem \ref{ramero2} conjoined with Theorem \ref{main11}. Differently
from Theorem \ref{main1}, we need to assume $K$ of characteristic $0$ to guarantee
the separability required in the statement of Theorem \ref{main1}.
\begin{proof}[Proof of Theorem \ref{main3}]
By assumption, $V_T$ is a $K(\un{T})$-variety so the ideal $\tilde{\id{p}}\coloneqq\langle
P_1,\ldots,P_l\rangle$ is a prime ideal of $K(\un{T})[\un{Y}]$. Equivalently, $\id{p}\coloneqq\langle P_1,\ldots,P_l\rangle$ is a prime
ideal of $K[\un{T}]$ and $\id{p}\cap K[\un{T}]=\{0\}$.\par
For $i=1,\ldots,r$, let $\mathcal{Q}_{S_i,R_i}$ be the quasi-generic polynomials defined at
the beginning of the section, i.e.
\[\mathcal{Q}_{S_i,R_i}(\un{\Lambda}_i,\un{T},\un{Y})=\sum_{j=1}^{N_{D_i}}\Lambda_{i,j}\hspace{1pt}Q_j(\un{Y})-T_i
=\mathcal{U}_{D_i}(\un{\Lambda}_i,\un{Y})-T_i\text{.}\]\par
The polynomials $\mathcal{Q}_{S_i,R_i}$ and $\id{p}$ satisfy the hypotheses of Theorem \ref{ramero2}, so
the ideal $\id{P}_{\un{S}}=\langle\id{p},\mathcal{Q}_{S_1,R_1},\ldots,\mathcal{Q}_{S_r,R_r}\rangle$ is a prime ideal of
$K[\un{\un{\Lambda}},\un{T},\un{Y}]$ such that $\id{P}_{\un{S}}\cap K[\un{\un{\Lambda}}]=\{0\}$ and
$\height\id{P}_{\un{S}}=\height\id{p}+r$.\par Denote by $\tilde{\id{P}}_{\un{S}}$ the extension of
$\id{P}_{\un{S}}$ to $K(\un{\un{\Lambda}})[\un{T},\un{Y}]$. By the classical bijective correspondence the ideal
$\tilde{\id{P}}_{\un{S}}$ is prime.\par
For $i=1,\ldots,r$, denote by $\pi_i$ the quotient morphism by the ideal $\langle\mathcal{Q}_{S_i,R_i}\rangle$.
Denote by $\tilde{\id{P}}_{\Lambda}$ the ideal of $K(\un{\un{\Lambda}})[\un{Y}]$ obtained by replacing $T_i$ with $\mathcal{U}_{D_i}$ for
every $i$.
Applying, recursively, all the morphisms $\pi_i$ to the ideal $\tilde{\id{P}}_{\un{S}}$, in the same manner
as for (\ref{height}), we obtain that $\tilde{\id{P}}_{\Lambda}$ is a prime ideal of $K(\un{\un{\Lambda}})[\un{Y}]$ and
\[\height\tilde{\id{P}}_{\Lambda}=\height\id{p}.\]\par
By bijective correspondence, the ideal
\[\id{P}_{\Lambda}=\langle P_1(\un{\mathcal{U}}(\un{\un{\Lambda}},\un{Y}),\un{Y}),\ldots,P_l(\un{\mathcal{U}}(\un{\un{\Lambda}},\un{Y}),\un{Y})\rangle,\]
where $\un{\mathcal{U}}(\un{\un{\Lambda}},\un{Y})=(\mathcal{U}_{D_1}(\un{\Lambda}_1,\un{Y}),\ldots,\mathcal{U}_{D_r}(\un{\Lambda}_r,\un{Y}))$,
is a prime ideal of $K[\un{\un{\Lambda}},\un{Y}]$ such that $\id{P}_{\Lambda}\cap K[\un{\un{\Lambda}}]=\{0\}$ and $\height \id{P}_{\Lambda}=\height\id{p}$.\par
Finally, we can apply Theorem \ref{main11} to $\id{P}_{\Lambda}$. For $\un{\un{\lambda}}$ in $K^{D_1+\ldots+D_r}$, consider the ideal
\[\id{P}_{\lambda}=\langle P_1(\un{\mathcal{U}}(\un{\un{\lambda}},\un{Y}),\un{Y}),\ldots,P_l(\un{\mathcal{U}}(\un{\un{\lambda}},\un{Y}),\un{Y})\rangle
=\langle P_1(\un{U}(\un{Y}),\un{Y}),\ldots,P_l(\un{U}(\un{Y}),\un{Y})\rangle\]
where $\un{\mathcal{U}}(\un{\un{\lambda}},\un{Y})=(\mathcal{U}_1(\un{\lambda}_1,\un{Y}),\ldots,\mathcal{U}_r(\un{\lambda}_r,\un{Y}))$ and $\un{U}(\un{Y})=(U_1(\un{Y}),\ldots,U_r(\un{Y}))$ with $U_i(\un{Y})=\mathcal{U}_i(\un{\lambda}_i,\un{Y})$. By Theorem \ref{main11}, for every
$\un{\un{\lambda}}$ in some Hilbert subset of $K^{D_1+\ldots+D_r}$,
the ideal $\id{P}_{\lambda}$ is prime and has height $\height\id{p}$, i.e. $V_U=V_K(\id{P}_{\lambda})$ is an irreducible $K$-variety and
\[\dim_K V_U=\dim_K\left(\faktor{K[\un{Y}]}{\id{P}_{\lambda}}\right)=s-\height\id{p}=d\text{.}\]\par
Recalling the isomorphism between $\A_K^{D_1+\ldots+D_r}$ and $\prod_{i=1}^r K[\un{\Lambda}_i]_{D_i}$
that we mentioned in the Introduction, taking a Hilbert subset of $K^{D_1+\ldots+D_r}$ is equivalent
to taking a Hilbert subset of $\prod_{i=1}^r K[\un{\Lambda}_i]_{D_i}$.
\end{proof}
\begin{oss}\label{3imp1}
As we mentioned in Remark \ref{rem3}(b), taking $D_i=0$, for every $i$, implies Theorem \ref{main1} in characteristic $0$. Indeed, for every $i=1,\ldots,r$, take \[\mathcal{Q}_{S_i,R_i}=\Lambda_{i,1}-T_i.\] The map $\phi_{R_i,S_i}$ sends $0$, the only point of $K^0$, to
$-T_i$. The element $-T_i$ is clearly transcendental over $K$ and, by hypothesis, $-T_i$ is not in $\id{p}$, so
$\phi_{R_i,S_i}$ is not identically $0$ for every $i$. Therefore, we apply Theorem \ref{main3} and
Theorem \ref{main1} follows.
\end{oss}
\bigskip

\bibliographystyle{alpha}
\bibliography{biblio}
\bigskip

\end{document}